\documentclass{amsart}
\usepackage{amssymb}
\usepackage{amsmath}
\usepackage{amscd}
\usepackage[left=3cm, right=2.5cm, top=2cm, bottom=2cm]{geometry}
\usepackage[dvipsnames]{xcolor}
\usepackage{hyperref}
\usepackage{mathtools}

\DeclarePairedDelimiter\floor{\lfloor}{\rfloor}

\newtheorem{prop}{Proposition}[section]

\newtheorem{stat}[prop]{Statement}
\newtheorem{lemma}[prop]{Lemma}
\newtheorem{theorem}[prop]{Theorem}

\theoremstyle{definition}

\newtheorem{remar}[prop]{Remark}

\renewcommand{\Im}{{\mathrm {Im}}}
\newcommand{\ord}{{\mathrm {ord}}}

\newcommand{\Z}{{\mathbb Z}}

\DeclareFontEncoding{OT2}{}{} 

\def\dis{\displaystyle}

\begin{document}

\title{On the vanishing of coefficients of $\eta^{26}$}

\author[]{}
\email{}
\address{}

\author[]{}
\email{}
\address{}

\author{S. Krishnamoorthy and T. Dalal  }
\address{S. Krishnamoorhty\\Indian Institute of Science Education and Research, Thiruvananthapuram, India} \email{srilakshmi@iisertvm.ac.in}
\address{Tarun Dalal\\Institute of Mathematical Sciences, ShanghaiTech University 393 Middle Huaxia Road, Pudong, Shanghai 201210, China} 
\email{tarun.dalal80@gmail.com}



\begin{abstract}
J.-P. Serre, in his paper \cite{1}, established a sufficient condition on $n$ for the $n$-th coefficient of the series $\eta^{26}$ to vanish. However, the question that whether this is a necessary condition remained unanswered. 
In this paper, using the theory of Hecke eigenforms explored by Serre, we prove some partial cases for the converse part.

\end{abstract}

\maketitle

\section{\textbf{Introduction}}

The Dedekind's eta-function is defined by
\begin{align*}
    \eta(z):=q^{1/24}\prod_{m=1}^\infty(1-q^m),
\end{align*}
where $q=e^{2\pi iz}$, $\Im(z)>0$. For an integer $r$, the $r$-th power of $\eta$ can be expressed as the infinite sum
\begin{align*}
    \eta^r(z)=q^{r/24}\sum_{n=0}^\infty p_r(n)q^n,
\end{align*}
where the coefficients $p_r(n)$ are defined by the identity
\begin{align*}
    \prod_{m=1}^\infty (1-q^m)^r=\sum_{n=0}^\infty p_r(n)q^n.
\end{align*}

There has been much interest in $p_r(n)$, in particular in its vanishing. It turns out that, for certain values of the exponent $r$, the set of $n$ such that $p_r(n)\neq0$ has density zero. In that case, the series $\eta^r$ is said to be lacunary. This is the case for $r=1,3$ as shown by the identities of Euler and Jacobi, respectively (\cite{3}, Ch: XIX):
\begin{align*}
    \prod_{m=1}^\infty(1-q^m)&=\sum_{n=-\infty}^\infty(-1)^nq^{(3n^2+n)/2}=1-q-q^2+q^5+q^7-q^{12}-q^{15}+q^{22}+q^{26}-\cdots,\\
    \prod_{m=1}^\infty(1-q^m)^3&=\sum_{n=-\infty}^\infty(-1)^n(2n+1)q^{(n^2+n)/2}=1-3q+5q^3-7q^6+9q^{10}-11q^{15}+13q^{21}-\cdots.
\end{align*}
Computer calculations make it probable that there is no other odd values of $r$ for which the series $\eta^r$ is lacunary. On the other hand, for even $r$, Ramanujan \cite{4} proved that $\eta^r$ is lacunary for $r=2,4,6,8$. Later, it was proved (\cite{5,6,7}) that this is also the case for $r=10,14\text{ and }26$. Finally, Serre \cite{1} proved that these are the only cases for even $r$ such that $\eta^r$ is lacunary.

For $r=2,4,6,8,10\text{ and }14$, Serre \cite{1} also established necessary and sufficient conditions on $n$ so that $p_r(n)=0$. Furthermore, for $r=26$, he proved the following theorem:

\begin{theorem}\cite{1}\label{Serre theorem}
If $n$ satisfies one of the following conditions, then $p_{26}(n)=0$.
\begin{itemize}
    \item There exists a prime number $p\equiv3\pmod4$ whose highest power dividing $12n+13$ is odd; and there exists a prime number $p'\equiv2\pmod3$ whose highest power dividing $12n+13$ is odd.
    \item $12n+13$ is a square; and all of its prime factors are congruent to $11$ modulo $12$.
\end{itemize}
\end{theorem}

This theorem is also proved by Chan et. al. \cite{2} by establishing an explicit formula for $p_{26}(n)$. However, whether the converse is true or not remained unanswered. Let us denote the highest power of a prime number $p$ dividing an integer $n$ as $ord_p(n)$. 
Note that the following statement is the converse of Theorem~\ref{Serre theorem}.

\begin{stat}\label{thm1}
If $p_{26}(n)=0$, then $n$ must satisfy one of the following conditions:
\begin{itemize}
    \item[I.] There exists a prime number $p\equiv3\pmod4$ such that $ord_p(12n+13)$ is odd; and there exists a prime number $p'\equiv2\pmod3$ such that $ord_{p'}(12n+13)$ is odd.
    \item[II.] $12n+13$ is a square; and all of its prime factors are congruent to $11$ modulo $12$.
\end{itemize}
\end{stat}

Observe that Statement \ref{thm1} is equivalent to check the following statement:
\begin{stat} \label{negation statement}
If $n$ satisfies one of the following conditions, then $p_{26}(n)\ne 0$:
\begin{itemize}
\item[N1.] For every prime $p\equiv 3 \pmod 4$, $\mathrm{ord}_p(12n+13)$ is even and there exists a prime $p^\prime \not\equiv 11 \pmod {12}$ such that $\mathrm{ord}_{p^\prime}(12n+13)>0.$
\item[N2.] For every prime $p\equiv 2 \pmod 3$, $\mathrm{ord}_p(12n+13)$ is even and there exists a prime $p^\prime \not\equiv 11 \pmod {12}$ such that $\mathrm{ord}_{p^\prime}(12n+13)>0.$ 
\end{itemize}
\end{stat}

In this article, we first verify some partial cases for N1 and N2. Finally as applications of those results, we prove the following necessary and sufficient conditions.

\begin{theorem}
\label{MT 1}
Let $n\in \mathbb{N}$ such that $\ord_p(12n+1)\not\equiv 4 \pmod 5$ for every prime $p\equiv 1 \pmod {12}$. Then
$p_{26}(25n+1)=0$ if and only if there exists a prime number $p\equiv 3 \pmod 4$ such that $\ord_p(12n+1)$ is odd; and there exists a prime number $p^\prime \equiv 2 \pmod 3$ such that $\ord_{p^\prime}(12n+1)$ is odd.
\end{theorem}

\begin{theorem}
\label{MT 2}
Let $n\in \mathbb{N}$ such that $\ord_p(12n+1)\not\equiv 6 \pmod 7$ for every prime $p\equiv 1 \pmod {12}$. Then
$p_{26}(49n+3)=0$ if and only if there exists a prime number $p\equiv 3 \pmod 4$ such that $\ord_p(12n+1)$ is odd; and there exists a prime number $p^\prime \equiv 2 \pmod 3$ such that $\ord_{p^\prime}(12n+1)$ is odd.
\end{theorem}



\section{\textbf{Preliminaries}}

Let us consider an imaginary quadratic field $K=\mathbb Q(\sqrt d)$ with discriminant $d$. We denote the ring of integers of $K$ by $\mathcal O_K$. Suppose, $c$ is a Hecke character on $K$ of exponent $k-1$ and conductor $\mathfrak f_c$, where $\mathfrak f_c$ is a nonzero ideal of $\mathcal O_K$.

Now we define the series
\begin{align*}
    \phi_{K,c}(z):=\sum_{\mathfrak a}c(\mathfrak a)q^{N(\mathfrak a)}\quad(q:=e^{2\pi iz}, \Im(z)>0),
\end{align*}
where the sum is taken over all ideals $\mathfrak a$ of $\mathcal O_K$ which are coprime to $\mathfrak f_c$. We have that $\phi_{K,c}(z)$ is a normalized Hecke eigenform of weight $k$, level $|d|\cdot N(\mathfrak f_c)$ and character $\epsilon_K\omega_c$, where $\epsilon_K$ and $\omega_c$ are Dirichlet characters defined as
\begin{align*}
    \epsilon_K(p)=\left(\frac{d}{p}\right)\text{ for primes }p\text{ such that }p\nmid2d
\end{align*}
and
\begin{align*}
    \omega_c(n)=\frac{c(n\mathcal O_K)}{n^{k-1}}\text{ for integers }n\text{ such that the ideal }n\mathcal O_K\text{ is coprime to }\mathfrak f_c.
\end{align*}
For more details please refer to \cite{1}.

Firstly, consider the number field $K_1=\mathbb Q(\sqrt{-3})$ of discriminant $-3$ and the ideal $\mathfrak f_1=4\sqrt{-3}\cdot\mathcal O_{K_1}$ of norm $48$. We now define Hecke characters $c_{1\pm}$ on $K_1$. Let $\mathfrak a$ be an ideal of $\mathcal O_{K_1}$ coprime to $\mathfrak f_1$. Since $\mathcal O_{K_1}$ is a principal ideal domain, we may choose a generator $a=x+y\sqrt{-3}$ of $\mathfrak a$ with $x,y\in\mathbb Z$, $x+y\equiv1\pmod2$ and $x\equiv1\pmod3$. We define
\begin{align*}
    c_{1\pm}(\mathfrak a)=(-1)^{(x\mp y-1)/2}a^{12}.
\end{align*}
It can be easily verified that $c_{1\pm}$ thus defined are Hecke characters on $K_1$ of exponent $12$ and conductor $\mathfrak f_1$, and that $\epsilon_{K_1}\omega_{c_{1+}}=\left(\frac{-1}{\bullet}\right)=\epsilon_{K_1}\omega_{c_{1-}}$.

Next, consider the number field $K_2=\mathbb Q(i)$ of discriminant $-4$ and the ideal $\mathfrak f_2=6\cdot\mathcal O_{K_2}$ of norm $36$. We now define Hecke characters $c_{2\pm}$ on $K_2$. Let $\mathfrak b$ be an ideal of $\mathcal O_{K_2}$ coprime to $\mathfrak f_2$. Since $\mathcal O_{K_2}$ is a principal ideal domain, we may choose a generator $b$ of $\mathfrak b$ and determine $m\in\mathbb Z/2\mathbb Z$ and $n\in\mathbb Z/8\mathbb Z$ such that
\begin{align*}
    b\equiv i^m\pmod{2\cdot\mathcal O_2} \quad \text{and} \quad b\equiv(1-i)^n\pmod{3\cdot\mathcal O_2}.
\end{align*}
We define
\begin{align*}
    c_{2\pm}(\mathfrak b)=(-1)^{3m}(\pm i)^{3n}b^{12}.
\end{align*}
It can be easily verified that $c_{2\pm}$ thus defined are Hecke characters on $K_2$ of exponent $12$ and conductor $\mathfrak f_2$, and that $\epsilon_{K_2}\omega_{c_{2+}}=\left(\frac{-1}{\bullet}\right)=\epsilon_{K_2}\omega_{c_{2-}}$.

By \cite{1}, we have
\begin{align*}
    \eta^{26}(12z)=\frac{1}{32617728}\left(\phi_{K_1,c_{1+}}(z)+\phi_{K_1,c_{1-}}(z)-\phi_{K_2,c_{2+}}(z)-\phi_{K_2,c_{2-}}(z)\right),
\end{align*}
and so $\eta^{26}(12z)$ is a linear combination of forms of \textit{CM}-type of weight $13$, level $144$ and character $\left(\frac{-1}{\bullet}\right)$. We write
\begin{align*}
    \phi_{K_1,c_{1\pm}}(z)=\sum_{n=1}^\infty t_{1\pm}(n)q^n \quad \text{and} \quad \phi_{K_2,c_{2\pm}}(z)=\sum_{n=1}^\infty t_{2\pm}(n)q^n.
\end{align*}
Thus we have
\begin{align*}
    \sum_{n=0}^\infty p_{26}(n)q^{12n+13}=\frac{1}{32617728}\left(\sum_{n=1}^\infty\big(t_{1+}(n)+t_{1-}(n)-t_{2+}(n)-t_{2-}(n)\big)q^n\right).
\end{align*}
Comparing the coefficients of the terms involving $q^{12n+13}$, we get
\begin{align}
p_{26}(n)=\frac{1}{32617728}\big(t_{1+}(12n+13)+t_{1-}(12n+13)-t_{2+}(12n+13)-t_{2-}(12n+13)\big)\label{eq:5}
\end{align}
for every nonnegative integer $n$.

Let the prime decomposition of $12n+13$ be given by
$$12n+13=\prod_{prime \ p} p^{\alpha_p}, \ \mathrm{where} \ \alpha_p:=ord_p(12n+13).$$
In view of the properties of Hecke eigenforms, we have
\begin{align}
    t_{1\pm}(12n+13)=\prod_p t_{1\pm}(p^{\alpha_p}) \quad \text{and} \quad t_{2\pm}(12n+13)=\prod_p t_{2\pm}(p^{\alpha_p}),\label{eq:8}
\end{align}
 Let us have a look at $t_{1+}(p^\alpha)$, $t_{1-}(p^\alpha)$, $t_{2+}(p^\alpha)$, $t_{2-}(p^\alpha)$ for a prime $p\geq5$. Let $t$ denote any one among $t_{1\pm}$ and $t_{2\pm}$. 
By the properties of Hecke eigenforms, 
if $p\equiv1\pmod4$, then
\begin{equation}
\label{Hecke relation for 1 mod 4}
t(p^r)=t(p)t(p^{r-1})-p^{12}t(p^{r-2})
\end{equation}
 for any integer $r\geq 2$. Equivalently 
 for any $\beta\geq0$,
\begin{align}
    t(p^{2\beta})&=t(p)^{2\beta}-\binom{2\beta-1}{1}p^{12}t(p)^{2\beta-2}+\dots+(-1)^{\beta-1}\binom{\beta+1}{2}p^{12\beta-12}t(p)^2+(-1)^\beta p^{12\beta},\label{eq:1}\\
    t(p^{2\beta+1})&=t(p)^{2\beta+1}-\binom{2\beta}{1}p^{12}t(p)^{2\beta-1}+\dots+(-1)^{\beta-1}\binom{\beta+2}{3}p^{12\beta-12}t(p)^3+(-1)^\beta\binom{\beta+1}{1}p^{12\beta}t(p).\label{eq:2}
\end{align}
On the other hand if $p\equiv3\pmod4$ then 
\begin{equation}
\label{Hecke relation for 3 mod 4}
t(p^r)=t(p)t(p^{r-1})+p^{12}t(p^{r-2})
\end{equation}
 for any integer $r\geq 2$. Equivalently 
for any $\beta\geq0$,
\begin{align}
    t(p^{2\beta})&=t(p)^{2\beta}+\binom{2\beta-1}{1}p^{12}t(p)^{2\beta-2}+\dots+\binom{\beta+1}{2}p^{12\beta-12}t(p)^2+p^{12\beta},\label{eq:3}\\
    t(p^{2\beta+1})&=t(p)^{2\beta+1}+\binom{2\beta}{1}p^{12}t(p)^{2\beta-1}+\dots+\binom{\beta+2}{3}p^{12\beta-12}t(p)^3+\binom{\beta+1}{1}p^{12\beta}t(p).\label{eq:4}
\end{align}

Firstly, let $p\equiv11\pmod{12}$. Since $p\equiv2\pmod3$ and $p\equiv3\pmod4$, therefore $p$ is inert in both $K_1$ and $K_2$. Consequently, $t_{1\pm}(p)=0=t_{2\pm}(p)$. By (\ref{eq:3}) and (\ref{eq:4}), we have
\begin{align}
    t_{1\pm}(p^\alpha)=t_{2\pm}(p^\alpha)=
    \begin{cases}
    p^{6\alpha} &\text{ if $\alpha$ is even,}\\
    0 &\text{ if $\alpha$ is odd.}
    \end{cases}\label{eq:6}
\end{align}

Secondly, let $p\equiv7\pmod{12}$. Since $p\equiv3\pmod4$, therefore $p$ is inert in $K_2$. Consequently, $t_{2\pm}(p)=0$. By (\ref{eq:3}) and (\ref{eq:4}), we have
\begin{align}
    t_{2\pm}(p^\alpha)=
    \begin{cases}
    p^{6\alpha} &\text{ if $\alpha$ is even,}\\
    0 &\text{ if $\alpha$ is odd.}
    \end{cases}\label{eq:7}
\end{align}
Since $p\equiv1\pmod3$, therefore we can write $p=z^2+3w^2=(z+w\sqrt{-3})(z-w\sqrt{-3})$ for some $z,w\in\mathbb Z$ with $z\equiv1\pmod3$. Moreover, as $z^2+3w^2\equiv3\pmod4$, we must have $z$ even and $w$ odd and thus we can choose $z,w$ such that $z+w\equiv1\pmod4$ and $z-w\equiv3\pmod4$. Hence, we have
\begin{align}
    t_{1+}(p)=-(z+w\sqrt{-3})^{12}+(z-w\sqrt{-3})^{12}=-t_{1-}(p).\label{eq:15}
\end{align}
By (\ref{eq:4}), for any odd $\alpha$, $t_{1+}(p^\alpha)=-t_{1-}(p^\alpha)$. By (\ref{eq:3}), for any even $\alpha$, $t_{1+}(p^\alpha)=t_{1-}(p^\alpha)$.

Thirdly, let $p\equiv5\pmod{12}$. Since $p\equiv2\pmod3$, therefore $p$ is inert in $K_1$. Consequently, $t_{1\pm}(p)=0$. By (\ref{eq:1}) and (\ref{eq:2}), we have
\begin{align}
    t_{1\pm}(p^\alpha)=
    \begin{cases}
    (-1)^{\alpha/2}p^{6\alpha} &\text{ if $\alpha$ is even,}\\
    0 &\text{ if $\alpha$ is odd.}
    \end{cases}\label{eq:9}
\end{align}
Since $p\equiv1\pmod4$, therefore we can write $p=x^2+y^2=(x+iy)(x-iy)$ for some $x,y\in\mathbb Z$ with $x$ odd and $y$ even. Moreover, as $x^2+y^2\equiv2\pmod3$, we must have $x\not\equiv0\pmod3$ and $y\not\equiv0\pmod3$ and thus we can choose $x,y$ such that $x\equiv1\equiv-y\pmod3$. Hence, we have
\begin{align}
    t_{2+}(p)=-i(x+iy)^{12}+i(x-iy)^{12}=-t_{2-}(p).\label{eq:16}
\end{align}
By (\ref{eq:2}), for any odd $\alpha$, $t_{2+}(p^\alpha)=-t_{2-}(p^\alpha)$. By (\ref{eq:1}), for any even $\alpha$, $t_{2+}(p^\alpha)=t_{2-}(p^\alpha)$.

Finally, let $p\equiv1\pmod{12}$. Since $p\equiv1\pmod3$, therefore we can write $p=z^2+3w^2=(z+w\sqrt{-3})(z-w\sqrt{-3})$ for some $z,w\in\mathbb Z$ with $z\not\equiv0\pmod3$. Moreover, as $z^2+3w^2\equiv1\pmod4$, we must have $z$ odd and $w$ even and thus we have $z+w\equiv z-w\pmod4$. Hence, we have
\begin{align}
    t_{1\pm}(p)=
    \begin{cases}
    (z+w\sqrt{-3})^{12}+(z-w\sqrt{-3})^{12} &\text{if $z+w\equiv1\equiv z-w\pmod4$,}\\
    -(z+w\sqrt{-3})^{12}-(z-w\sqrt{-3})^{12} &\text{if $z+w\equiv3\equiv z-w\pmod4$.}
    \end{cases}\label{eq:17}
\end{align}
Since $p\equiv1\pmod4$, therefore we can write $p=x^2+y^2=(x+iy)(x-iy)$ for some $x,y\in\mathbb Z$ with $x$ odd and $y$ even. Moreover, as $x^2+y^2\equiv1\pmod3$, we can choose $x,y$ such that either $x\not\equiv0\pmod3$, $y\equiv0\pmod3$ or $x\equiv0\pmod3$, $y\not\equiv0\pmod3$. Hence, we have
\begin{align}
    t_{2\pm}(p)=
    \begin{cases}
    (x+iy)^{12}+(x-iy)^{12} &\text{if }x\not\equiv0\pmod3\text{ , }y\equiv0\pmod3,\\
    -(x+iy)^{12}-(x-iy)^{12} &\text{if }x\equiv0\pmod3\text{ , }y\not\equiv0\pmod3.
    \end{cases}\label{eq:10}
\end{align}
From (\ref{eq:5}), it is clear that $t_{1+}(p)+t_{1-}(p)-t_{2+}(p)-t_{2-}(p)\equiv0\pmod{3}$. Note that
\begin{align*}
    (z+w\sqrt{-3})^{12}+(z-w\sqrt{-3})^{12}+(x+iy)^{12}+(x-iy)^{12}\equiv2(x^{12}+z^{12})\not\equiv0\pmod3.
\end{align*}
Thus, we must have $x\equiv1\pmod3$, $y\equiv0\pmod3$ when $z+w\equiv1\equiv z-w\pmod4$ and $x\equiv0\pmod3$, $y\equiv1\pmod3$ when $z+w\equiv3\equiv z-w\pmod4$. In other words, the same sign appears on both $t_{1+}(p)$ and $t_{2+}(p)$.

\section{\textbf{Some Important Lemmas}}

In this section, we study some properties of $t_{1\pm}$ and $t_{2\pm}$.
\subsection{Divisibility properties of $t_{1\pm}$ and $t_{2\pm}$}
\begin{theorem}
\label{theorem 1 appendix}
Let $p\equiv 1 \pmod {12}$ be a prime. If $t_{2+}(p)=2 \pmod 5$, then 
\begin{equation*}
t_{2+}(p^{5l+n})\equiv t_{2+}(p^{5(l-1)+n}) \pmod 5 \ \mathrm{for} \ l\in \mathbb{N} \ \mathrm{and} \ n\geq 0.
\end{equation*}
\end{theorem}
\begin{proof}
Recall that we have the Hecke relation (cf. \eqref{Hecke relation for 1 mod 4})
\begin{equation}
\label{eq 4200}
t_{2+}(p^r)\equiv 2t_{2+}(p^{r-1})-t_{2+}(p^{r-2}) \pmod 5 \ \mathrm{for} \ r\geq 2.
\end{equation}
Let $l\in \mathbb{N}$ and $n\geq 0$.
By repeated application of \eqref{eq 4200}, we get
\begin{align*}
t_{2+}(p^{5l+n})&\equiv 2t_{2+}(p^{5l+n-1})-t_{2+}(p^{5l+n-2}) \pmod 5\\
&\equiv 2(2t_{2+}(p^{5l+n-2})-t_{2+}(p^{5l+n-3}))- t_{2+}(p^{5l+n-2}) \pmod 5\\
&\equiv 3t_{2+}(p^{5l+n-2})-2t_{2+}(p^{5l+n-3}) \pmod 5\\
&\equiv 3(2t_{2+}(p^{5l+n-3})-t_{2+}(p^{5l+n-4}))- 2t_{2+}(p^{5l+n-3}) \pmod 5\\
&\equiv 4t_{2+}(p^{5l+n-3})-3t_{2+}(p^{5l+n-4}) \pmod 5\\
&\equiv 4(2t_{2+}(p^{5l+n-4})-t_{2+}(p^{5l+n-5}))- 3t_{2+}(p^{5l+n-4}) \pmod 5\\
&\equiv 5t_{2+}(p^{5l+n-4})-4t_{2+}(p^{5l+n-5}) \pmod 5\\
&\equiv t_{2+}(p^{5(l-1)+n}) \pmod 5.
\end{align*}
This completes the proof.
\end{proof}
\begin{remar}
Observe that a similar argument as in the proof of Theorem \ref{theorem 1 appendix} shows that for any prime $p\equiv 1 \pmod {12}$, if $t_{2+}(p)=3 \pmod 5$, then 
$t_{2+}(p^{5l+n})\equiv - t_{2+}(p^{5(l-1)+n}) \pmod 5 \ \mathrm{for} \ l\in \mathbb{N} \ \mathrm{and} \ n\geq 0.$
\end{remar}

\begin{theorem}
\label{theorem 2 appendix}
Let $p\equiv 1 \pmod {12}$ be a prime. If $t_{1+}(p)=2 \pmod 7$, then 
\begin{equation*}
t_{1+}(p^{7l+n})\equiv t_{1+}(p^{7(l-1)+n}) \pmod 7 \ \mathrm{for} \ l\in \mathbb{N} \ \mathrm{and} \ n\geq 0.
\end{equation*}
\end{theorem}
\begin{proof}
Recall that we have the Hecke relation (cf. \eqref{Hecke relation for 1 mod 4})
\begin{equation}
\label{eq 4300}
t_{1+}(p^r)\equiv 2t_{1+}(p^{r-1})-t_{1+}(p^{r-2}) \pmod 7 \ \mathrm{for} \ r\geq 2.
\end{equation}
Let $l\in \mathbb{N}$ and $n\geq 0$.
By repeated application of \eqref{eq 4300}, we get
\begin{align*}
t_{1+}(p^{7l+n})&\equiv 2t_{1+}(p^{7l+n-1})-t_{1+}(p^{7l+n-2}) \pmod 7\\
&\equiv 2(2t_{1+}(p^{7l+n-2})-t_{1+}(p^{7l+n-3}))- t_{1+}(p^{7l+n-2}) \pmod 7\\
&\equiv 3t_{1+}(p^{7l+n-2})-2t_{1+}(p^{7l+n-3}) \pmod 7\\
&\equiv 3(2t_{1+}(p^{7l+n-3})-t_{1+}(p^{7l+n-4}))- 2t_{1+}(p^{7l+n-3}) \pmod 7\\
&\equiv 4t_{1+}(p^{7l+n-3})-3t_{1+}(p^{7l+n-4}) \pmod 7\\
&\equiv 4(2t_{1+}(p^{7l+n-4})-t_{1+}(p^{7l+n-5}))- 3t_{1+}(p^{7l+n-4}) \pmod 7\\
&\equiv 5t_{1+}(p^{7l+n-4})-4t_{1+}(p^{7l+n-5}) \pmod 7\\
&\equiv 5(2t_{1+}(p^{7l+n-5})-t_{1+}(p^{7l+n-6})) - 4t_{1+}(p^{7l+n-5}) \pmod 7\\
&\equiv 6t_{1+}(p^{7l+n-5})-5t_{1+}(p^{7l+n-6}) \pmod 7\\
&\equiv 6(2t_{1+}(p^{7l+n-6})-t_{1+}(p^{7l+n-7})) -5t_{1+}(p^{7l+n-6})\pmod 7\\
&\equiv 7t_{1+}(p^{7l+n-6})-6t_{1+}(p^{7l+n-7}) \pmod 7\\
&\equiv t_{1+}(p^{7(l-1)+n}) \pmod 7.
\end{align*}
This completes the proof.
\end{proof}
\begin{remar}\label{theorem 2 appendix remark}
Observe that a similar argument as in the proof of Theorem \ref{theorem 2 appendix} shows that for any prime $p\equiv 1 \pmod {12}$, if $t_{1+}(p)=5 \pmod 7$, then 
$t_{1+}(p^{7l+n})\equiv - t_{1+}(p^{7(l-1)+n}) \pmod 7 \ \mathrm{for} \ l\in \mathbb{N} \ \mathrm{and} \ n\geq 0.$
\end{remar}

\begin{prop}
\label{t2+ prime = 5 mod 12}
For any prime $p\equiv 5 \pmod {12}$ we have
\begin{enumerate}
\item [(i)] $5\nmid t_{2+}(5)$, $5|t_{2+}(p)$ for $p\ne 5$. Hence $5\nmid t_{2+}(p^{2\alpha})$ for any prime $p \equiv 5 \pmod {12}$.
\item [(ii)] $t_{2+}(p)\equiv 0,2,5 \pmod 7$. Furthermore, $7|t_{2+}(p)$ if and only if $p\equiv 1,2,4 \pmod 7$.
\end{enumerate}

\end{prop}
\begin{proof}
\begin{enumerate}
\item [(i)]
From \eqref{eq:16} we get $t_{2+}(5)=20592$ (note that we can choose $x=1,y=2$). Thus $5\nmid t_{2+}(5)$. Consequently, from \eqref{eq:1} we conclude that $5\nmid t_{2+}(5^{2\alpha})$ for all $\alpha\in \mathbb{N}$.

Now let $p\ne 5$ be a prime such that $p\equiv 5 \pmod {12}$. From \eqref{eq:16} we have
\begin{equation}
\label{expansion of t2+ for p = 5 mod 12}
t_{2+}(p)=-i(x+iy)^{12}+i(x-iy)^{12}=2(12 x^{11}y-220x^9y^3+792x^7y^5-792x^5y^7+220x^3y^9-12xy^{11}),
\end{equation}
where $p=x^2+y^2$. Clearly, $(x,y)\not\equiv (0,0)\pmod 5$. 
If $xy\equiv 0 \pmod 5$, then \eqref{expansion of t2+ for p = 5 mod 12} implies $5|t_{2+}(p)$. Since $5\nmid p$, from \eqref{eq:1} we conclude that $5\nmid t_{2+}(p^{2\alpha})$ for all $\alpha\in \mathbb{N}$.

Now assume that $xy\not\equiv 0 \pmod 5$. This assumption clearly forces $x^2\equiv \pm 1 \pmod 5$ and $y^2\equiv \pm 1 \pmod 5$ (since $a^4\equiv 1 \pmod 5$ for all $a\in (\Z/5\Z)\char`\\ \{0\}$). Observe that if $x^2y^2\equiv -1 \pmod 5$, then $p=x^2+y^2 \equiv 0 \pmod 5$. This contradicts the fact that $p>5$ is a prime. Hence $x^2y^2\equiv 1 \pmod 5$. Consequently, $(x^2-y^2)\equiv 0\pmod 5$. 
Now using \eqref{expansion of t2+ for p = 5 mod 12}, we can write
\begin{align*}
t_{2+}(p)&\equiv -x^{11}y-x^7y^5+x^5y^7+xy^{11} \pmod 5\\
&\equiv -(x^2-y^2)(x^9y+x^7y^3+2x^5y^5+x^3y^7+xy^9) \pmod 5.
\end{align*}
 This shows that $5|t_{2+}(p)$. Since $5\nmid p$, from \eqref{eq:1} we conclude that $5\nmid t_{2+}(p^{2\alpha})$ for all $\alpha\in \mathbb{N}$.

\item [(ii)] 
Recall that $p=x^2+y^2 \equiv 5 \pmod {12}$ is a prime  and 
\begin{equation}
\label{expansion of t2+ for p = 5 mod 12 repeated}
t_{2+}(p)=-i(x+iy)^{12}+i(x-iy)^{12}=2(12 x^{11}y-220x^9y^3+792x^7y^5-792x^5y^7+220x^3y^9-12xy^{11}).
\end{equation}
Taking reduction modulo $7$, we can write
\begin{equation}
\label{expansion of t2+ for p = 5 mod 12 reduction mod 7}
t_{2+}(p)\equiv xy(x^2-y^2)(3x^8 + 4x^6y^2 + 6x^4y^4 + 4x^2y^6 + 3y^8) \pmod 7.
\end{equation}

By explicit computation (or using computer), it can be checked that for $x,y\in \Z/7\Z$, the polynomial in \eqref{expansion of t2+ for p = 5 mod 12 reduction mod 7} takes values only $0,2,5$ modulo $7$. 
For example, using the following code in \texttt{MAGMA} (\url{http://magma.maths.usyd.edu.au/calc/}), we can conclude that the polynomial in \eqref{expansion of t2+ for p = 5 mod 12 reduction mod 7} takes values only $0,2,5$ modulo $7$:
\begin{verbatim}
for x,y in GF(7) do
x*y*(x^2-y^2)*(3*x^8 + 4*x^6*y^2 + 6*x^4*y^4 + 4*x^2*y^6 + 3*y^8);
end for;
\end{verbatim}

In particular, $t_{2+}(p)$ modulo $7$ takes every value in $\{0,2,5\}$ at least once. For example, using \eqref{expansion of t2+ for p = 5 mod 12 repeated} it is easy to see that $t_{2+}(29)=t_{2+}(5^2+2^2)\equiv 0 \pmod 7,$ $t_{2+}(257)=t_{2+}(1+16^2)\equiv 5 \pmod 7$ and $t_{2+}(677)=t_{2+}(1+26^2)\equiv 2 \pmod 7$. Thus we conclude that $t_{2+}(p)\equiv 0,2,5 \pmod 7$.

Since $p\ne 7$ is a prime, we have $(x,y)\not \equiv (0,0)\pmod 7$. Observe that, the following is true
\begin{equation}
\label{eq 1900}
(x^{12}+y^{12})(x^6+y^6)\equiv (3x^8 + 4x^6y^2 + 6x^4y^4 + 4x^2y^6 + 3y^8)\cdot f(x,y) \pmod 7,
\end{equation}
where $f(x,y):= 5x^{10} + 5x^8y^2 + 2x^6y^4 + 2x^4y^6 + 5x^2y^8 +5y^{10}$. Since $a^6\equiv 1 \pmod 7$ for all $a\in (\Z/7\Z)\char`\\ \{0\}$, we have $(x^{12}+y^{12})(x^6+y^6)\not \equiv 0 \pmod 7$ for all $x,y\in \Z/7\Z$ with $(x,y)\not \equiv (0,0) \pmod 7$. Thus from \eqref{eq 1900} we get $$(3x^8 + 4x^6y^2 + 6x^4y^4 + 4x^2y^6 + 3y^8)\not \equiv 0 \pmod 7 \ \mathrm{for \ all} \ x,y\in \Z/7\Z \ \mathrm{with} \ (x,y)\not \equiv (0,0) \pmod 7.$$


Now from \eqref{expansion of t2+ for p = 5 mod 12 reduction mod 7} we see that $7|t_{2+}(p)$ if and only if $xy(x^2-y^2) \equiv 0 \pmod 7$, i.e, either $xy\equiv 0 \pmod 7$ or $(x^2-y^2) \equiv 0 \pmod 7$. The condition $xy \equiv 0 \pmod 7$ (i.e either $x\equiv 0 \pmod 7$ or $y \equiv 0 \pmod 7$) implies that $p=x^2+y^2\equiv 1,2,4 \pmod 7$. On the other hand, the condition $x^2-y^2\equiv 0 \pmod 7$ implies that $p=x^2+y^2\equiv 2y^2\equiv 1,2,4 \pmod 7$.

Conversely, let $p=x^2+y^2\equiv 1,2,4 \pmod 7$. Since  
$a^2\equiv 1,2,4 \pmod 7$ for any $a\in (\Z/7\Z)\char`\\ \{0\}$, we have
\begin{equation}
\label{eq 1919}
 x^2y^2(x^2+y^2)\not \equiv -1 \pmod 7.
 \end{equation}
Suppose that $xy(x^2-y^2)\not \equiv 0 \pmod 7$. In particular, this implies $x\not \equiv 0 \pmod 7, y\not\equiv 0 \pmod 7$ and $x^2-y^2 \not\equiv 0\pmod 7$.
Since $a^6\equiv 1 \pmod 7$ for all $a\in (\Z/7\Z)\char`\\ \{0\}$, we have 
\begin{equation}
\label{eq 2100}
1\equiv (x^2-y^2)^6 \equiv x^{12} + x^{10}y^2 + x^8y^4 + x^6y^6 + x^4y^8 + x^2y^{10} + y^{12} \equiv 3+2x^2y^2(x^2+y^2) \pmod 7.
\end{equation}
Since $gcd(2,7)=1$, from \eqref{eq 2100} we obtain
\begin{equation}
x^2y^2(x^2+y^2) \equiv -1 \pmod 7.
\end{equation}
This contradicts \eqref{eq 1919}. Therefore, we must have $xy(x^2-y^2)\equiv 0 \pmod 7$. Consequently, from \eqref{expansion of t2+ for p = 5 mod 12 reduction mod 7} we conclude that $7|t_{2+}(p)$. This completes the proof.
\end{enumerate}
\end{proof}

%

\begin{prop}
\label{t1+  prime = 7 mod 12}
For any prime $p\equiv 7 \pmod {12}$, we have 
\begin{enumerate}
\item [(i)] $t_{1+}(p)=5k\sqrt{-3}$ for some $k\in \Z$. Hence $5\nmid t_{1+}(p^{2\alpha})$ for every prime $p\equiv 7 \pmod {12}$.\label{prop 3.2 i}
\item [(ii)] $7\nmid \frac{t_{1+}(7)}{\sqrt{-3}}$, $7|\frac{t_{1+}(p)}{\sqrt{-3}}$ for $p\ne 7$. Hence $7\nmid t_{1+}(p^{2\alpha})$ for any prime $p \equiv 7 \pmod {12}$.\label{prop 3.2 ii}
\end{enumerate}
\end{prop}
\begin{proof}
Let $p \equiv 7 \pmod {12}$. From \eqref{eq:15}, we have 
\begin{equation} 
\label{expansion of t1+ prime p equiv 7 mod 12}
    t_{1+}(p)=-(x+y\sqrt{-3})^{12}+(x-y\sqrt{-3})^{12}= (-24 x^{11} y + 1320x^9y^3 - 14256x^7y^5 + 42768x^5y^7 - 35640x^3y^9 +5832xy^{11})\sqrt{-3},
\end{equation}
where $p=x^2+3y^2$. Let $h(x,y):= -24 x^{11} y + 1320x^9y^3 - 14256x^7y^5 + 42768x^5y^7 - 35640x^3y^9 +5832xy^{11}$.
\begin{enumerate}
\item [(i)]
Taking reduction modulo $5$, we get 
\begin{equation}
\label{eq 2400}
h(x,y)\equiv xy(x^4-y^4)(x^6 + 3y^6) \pmod 5.
\end{equation}
Since $a^4\equiv 1 \pmod 5$ for all $a\in \Z/5\Z\char`\\ \{0\}$, from \eqref{eq 2400} we obtain $h(x,y)\equiv 0 \pmod 5$.
Hence $t_{1+}(p)=5k\sqrt{-3}$ for some $k\in \Z$.
Since $5\nmid p$, from \eqref{eq:3} it follows that $5\nmid t_{1+}(p^{2\alpha})$ for $\alpha\in \mathbb{N}$.
\item [(ii)] 
Form \eqref{expansion of t1+ prime p equiv 7 mod 12} we have $t_{1+}(7)= -102960\sqrt{-3}$ (note that in this case we can choose $x=2, y=1$). Thus $7\nmid \frac{t_{1+}(7)}{\sqrt{-3}}$. Consequently, from \eqref{eq:3} we get $7\nmid t_{1+}(7^{2\alpha})$ for $\alpha\in \mathbb{N}$.
 
Now let $p=x^2+3y^2\ne 7$ be a prime such that $p\equiv 7 \pmod {12}$. Observe that $x^2+3y^2\not \equiv 0 \pmod 7$ and we have the relation
\begin{equation}
\label{eq 2500}
(x^2+3y^2)h(x,y)\equiv xy(x^6-y^6)(4x^6 + 2x^4y^2 + x^2y^4 + 4y^6) \pmod 7.
\end{equation}
Since $a^6\equiv 1 \pmod 7$ for all $a\in \Z/7\Z\char`\\ \{0\}$, we always have $xy(x^6-y^6)\equiv 0 \pmod 7$. Using the fact that $x^2+3y^2 \not \equiv 0 \pmod 7$, from \eqref{eq 2500} we obtain $h(x,y)\equiv 0 \pmod 7$. 
Hence we have $7|\frac{t_{1+}(p)}{\sqrt{-3}}$ for every prime $p\equiv 7 \pmod {12}$ and $p\ne 7$. Consequently, from \eqref{eq:3} we conclude that $7\nmid t_{1+}(p^{2\alpha})$ for $\alpha\in \mathbb{N}$.
\end{enumerate}
\end{proof}

\begin{prop}
\label{t2+, t1+ prime = 1 mod 12}
For any prime $p\equiv 1 \pmod {12}$, we have 
\begin{enumerate}
\item [(i)] $t_{2+}(p)\equiv 2 \ \mathrm{or} \ 3 \pmod 5$. Moreover, $5\nmid t_{2+}(p^\alpha)$ if and only if $\alpha \not\equiv 4 \pmod 5$.\label{prop 3.3 i}
\item [(ii)] $t_{2+}(p)\equiv 0,2,5 \pmod 7$. 
\item [(iii)] $t_{1+}(p) \equiv 2,5 \pmod 7$. Moreover, $7\nmid t_{1+}(p^\alpha)$ if and only if $\alpha \not\equiv 6 \pmod 7$.\label{prop 3.3 iii}
\item [(iv)] $t_{1+}(p)\equiv 2 \ \mathrm{or} \ 3 \pmod 5$. 

\end{enumerate}

\end{prop}
\begin{proof}
Let $p\equiv 1 \pmod {12}$.  
\begin{enumerate}
\item [(i)]
WLOG we assume that (cf. \eqref{eq:10})
\begin{equation}
t_{2+}(p)=(x+iy)^{12}+(x-iy)^{12}= 2\{x^{12}-66x^{10}y^2+495x^8y^4-924x^6y^6+495x^4y^8-66x^2y^{10}+y^{12}\},
\end{equation}
where $p=x^2+y^2$.
Observe that $(x,y)\not\equiv (0,0)\pmod 5$.
Taking reduction modulo $5$, we get 
\begin{equation}
\label{eq 2700}
t_{2+}(p)\equiv 2x^{12} + 3x^{10}y^2 + 2x^6y^6 + 3x^2y^{10} + 2y^{12} \pmod 5.
\end{equation}
First consider the case $xy\equiv 0 \pmod 5$, i.e either $x\equiv 0 \pmod 5$ or $y\equiv 0 \pmod 5$. Using the condition $(x,y)\not \equiv (0,0) \pmod 5$ and the fact that  $a^4\equiv 1 \pmod 5$ for all $a\in \Z/5\Z\char`\\ \{0\}$, we get $x^{12}+y^{12}\equiv 1 \pmod 5$. Now using the assumption $xy\equiv 0 \pmod 5$, from \eqref{eq 2700} we obtain $t_{2+}(p)\equiv 2 \pmod 5$.

Next assume that $xy\not \equiv 0 \pmod 5$.
A similar argument as in the proof of Proposition~\ref{t2+ prime = 5 mod 12}(i) shows that $x^2y^2 \equiv 1 \pmod 5$. 
Using the fact that  $a^4\equiv 1 \pmod 5$ for all $a\in \Z/5\Z\char`\\ \{0\}$, from \eqref{eq 2700} we get
\begin{align*}
t_{2+}(p)
&\equiv 2(1-x^2y^2-4x^2y^2-x^2y^2+1) \pmod 5\\
&\equiv 2 \pmod 5.
\end{align*}
On the other hand, considering the other value of $t_{2+}(p)$ from \eqref{eq:10} with negative sign and using a similar argument, we conclude that $t_{2+}(p)\equiv 2 \ \mathrm{or} \ 3 \pmod 5$.
This proves the first part of (i).

To prove the second part, first consider the case $t_{2+}(p)\equiv 2 \pmod 5$. 
Since $p^4 \equiv 1 \pmod 5$, from the Hecke relation \eqref{Hecke relation for 1 mod 4}, we get
\begin{equation}
\label{eq 2800}
t_{2+}(p^r)=t_{2+}(p)t_{2+}(p^{r-1})-p^{12}t_{2+}(p^{r-2})\equiv 2t_{2+}(p^{r-1})-t_{2+}(p^{r-2}) \pmod 5 \ \mathrm{for} \ r\geq 2.
\end{equation}
From \eqref{eq 2800}, we have (recall that $t_{2+}(1)=1$)
\begin{equation}
\label{2929}
t_{2+}(p^2)\equiv 3 \pmod 5, \ t_{2+}(p^3)\equiv 4 \pmod 5,  \ t_{2+}(p^4)\equiv 0 \pmod 5 \ \mathrm{and} \ t_{2+}(p^5)\equiv 1 \pmod 5.
\end{equation}
Furthermore, from Theorem \ref{theorem 1 appendix} we have
\begin{equation}
\label{2900}
t_{2+}(p^{5l+n})\equiv t_{2+}(p^{5(l-1)+n}) \pmod 5 \ \mathrm{for} \ l\in \mathbb{N} \ \mathrm{and} \ n\geq 0.
\end{equation}
 Now combining \eqref{2929} and \eqref{2900}, we conclude that $t_{2+}(p^\alpha)\equiv 0 \pmod 5$ if and only if $\alpha \equiv 4 \pmod 5$. Equivalently we have $5\nmid t_{2+}(p^\alpha)$ if and only if $\alpha \not\equiv 4 \pmod 5$.

A similar argument will work for the case $t_{2+}(p)\equiv 3 \pmod 5$. Note that in this case we have $t_{2+}(p^{5l+n})\equiv -t_{2+}(p^{5(l-1)+n}) \pmod 5$ for $l\in \mathbb{N}$. This completes the proof of (i).

\item [(ii)]
By explicit computation it can be verified that for $x,y\in \Z/7\Z$, the polynomial $$\pm((x+iy)^{12}+(x-iy)^{12})=\pm 2\{x^{12}-66x^{10}y^2+495x^8y^4-924x^6y^6+495x^4y^8-66x^2y^{10}+y^{12}\}$$ takes values only $0,2,5$ modulo $7$. This fact also can be verified via computer. For example, using the following code in \texttt{MAGMA} (\url{http://magma.maths.usyd.edu.au/calc/}), we can conclude that the polynomial $\pm((x+iy)^{12}+(x-iy)^{12})$ takes values only $0,2,5$ modulo $7$:
\begin{verbatim}
for x,y in GF(7) do
d:=2*(x^(12)-66*x^(10)*y^2+495*x^8*y^4-924*x^6*y^6+495*x^4*y^8-66*x^2*y^(10)+y^(12));
d, -d;
end for;
\end{verbatim}


\item [(iii)]
 WLOG we assume that (cf. \eqref{eq:17})
\begin{equation}
\label{3100}
t_{1+}(p)= (x+\sqrt{-3}y)^{12}+(x-\sqrt{-3}y)^{12}=2x^{12} - 396x^{10}y^2 + 8910x^8y^4 - 49896x^6y^6 + 80190x^4y^8 -32076x^2y^{10} + 1458y^{12},
\end{equation}
where $p=x^2+3y^2$. Observe that $x^2+3y^2\not \equiv 0 \pmod 7$ and $(x,y)\not\equiv (0,0)\pmod 7$. Taking reduction modulo $7$, we have
\begin{equation}
\label{3200}
t_{1+}(p)\equiv 2x^{12} + 3x^{10}y^2 + 6x^8y^4 + 5x^6y^6 + 5x^4y^8 + 5x^2y^{10} + 2y^{12} \pmod 7.
\end{equation}
First consider the case $xy\equiv 0 \pmod 7$, i.e either $x\equiv 0 \pmod 7$ or $y\equiv 0 \pmod 7$. Using the condition $(x,y)\not \equiv (0,0) \pmod 7$ and the fact that  $a^6\equiv 1 \pmod 7$ for all $a\in \Z/7\Z\char`\\ \{0\}$, we get $x^{12}+y^{12}\equiv 1 \pmod 7$. Now using the fact $xy\equiv 0 \pmod 7$, from \eqref{3200} we obtain $t_{1+}(p)\equiv 2 \pmod 7$.

Next assume that $xy\not \equiv 0 \pmod 7$. Using \eqref{3200}, it is easy to see that we have the following relation
\begin{equation}
\label{3300}
(x^2+3y^2)(t_{1+}(p)-2x^6y^6)\equiv (x^6-y^6)(2x^8 + 2x^6y^2 + x^4y^4 + 4x^2y^6 + y^8) \pmod 7.
\end{equation}
The assumption $xy\not\equiv 0 \pmod 7$ implies that $x^6\equiv y^6\equiv 1 \pmod 7$, i.e., $x^6-y^6\equiv 0 \pmod 7$. Since $x^2+3y^2 \not \equiv 0 \pmod 7$, from \eqref{3300} we conclude that 
$$t_{1+}(p)\equiv 2x^6y^6\equiv 2 \pmod 7.$$

On the other hand, considering the other value of $t_{1+}(p)$ from \eqref{eq:17} with negative sign and using a similar argument, we can conclude that $t_{1+}(p)\equiv 2 \ \mathrm{or} \ 5 \pmod 7$.
This proves the first part of (iii).

To prove the second part, first consider the case $t_{1+}(p)\equiv 2 \pmod 7$. 
Since $7\nmid p$, from the Hecke relation \eqref{Hecke relation for 1 mod 4}, we get
\begin{equation}
\label{3400}
t_{1+}(p^r)=t_{1+}(p)t_{1+}(p^{r-1})-p^{12}t_{1+}(p^{r-2})\equiv 2t_{1+}(p^{r-1})-t_{1+}(p^{r-2}) \pmod 7 \ \mathrm{for} \ r\geq 2.
\end{equation}
From \eqref{3400}, we have (recall that $t_{1+}(1)=1$)
\begin{equation}
\label{3500}
t_{1+}(p^i)\equiv i+1 \pmod 7 \ \mathrm{for} \ 2\leq i \leq 5, \ t_{1+}(p^6)\equiv 0 \pmod 7 \ \mathrm{and} \ t_{1+}(p^7) \equiv 1 \pmod 7.
\end{equation}
Furthermore, from Theorem \ref{theorem 2 appendix} we have
\begin{equation}
\label{3600}
t_{1+}(p^{7l+n})\equiv t_{1+}(p^{7(l-1)+n}) \pmod 7 \ \mathrm{for} \ l\in \mathbb{N}, \ \mathrm{and} \ n\geq 0.
\end{equation}
 Now combining \eqref{3500} and \eqref{3600}, we conclude that $t_{1+}(p^\alpha)\equiv 0 \pmod 7$ if and only if $\alpha \equiv 6 \pmod 7$. Equivalently we have $7\nmid t_{1+}(p^\alpha)$ if and only if $\alpha \not\equiv 6 \pmod 7$.

A similar argument will work for the case $t_{1+}(p)\equiv 5 \pmod 7$. Note that in this case we have $t_{1+}(p^{7l+n})\equiv -t_{1+}(p^{7(l-1)+n}) \pmod 7$ for $l\in \mathbb{N}$ and $n\geq 0$. This completes the proof of (iii).

%

\item [(iv)]
 First assume that (cf. \eqref{eq:17})
\begin{equation}
t_{1+}(p)= (x+\sqrt{-3}y)^{12}+(x-\sqrt{-3}y)^{12}=2x^{12} - 396x^{10}y^2 + 8910x^8y^4 - 49896x^6y^6 + 80190x^4y^8 -32076x^2y^{10} + 1458y^{12},
\end{equation}
where $p=x^2+3y^2$.
Observe that $(x,y)\not\equiv (0,0)\pmod 5$.
Taking reduction modulo $5$, we get 
\begin{equation}
\label{8400}
t_{1+}(p)\equiv 2x^{12} + 4x^{10}y^2 + 4x^6y^6 + 4x^2y^{10} +3y^{12} \pmod 5.
\end{equation}
If $xy \equiv 0 \pmod 5$, then using the facts that $(x,y)\not\equiv (0,0)\pmod 5$ and $a^4\equiv 1 \pmod 5$ for all $a\in \Z/5\Z\char`\\ \{0\}$, from \eqref{8400} we get $t_{1+}(p)\equiv 2,3 \pmod 5$.

Now assume that $xy\not \equiv 0 \pmod 5$. Using the fact that $a^2\equiv \pm 1 \pmod 5$ for all $a\in \Z/5\Z\char`\\ \{0\}$, from \eqref{8400}, it is easy to see that
$$t_{1+}\equiv 2+4.3(xy)^2+3\equiv \pm 2 \pmod 5.$$
Considering the other value of $t_{1+}(p)$ from \eqref{eq:17} with negative sign and using a similar argument, we can conclude that $t_{1+}(p)\equiv 2 \ \mathrm{or} \ 3 \pmod 5$. This completes the proof.
\end{enumerate}
\end{proof}
\subsection{Non-vanishing of $t_{1+}(p^\alpha)-t_{2+}(p^\alpha)$, when $\alpha>1$ is even}
\begin{lemma}\label{lemma1}
For a prime number $p\equiv1\pmod{12}$, we have $ord_2(t_{1+}(p)-t_{2+}(p))<\infty$ and so $t_{1+}(p)-t_{2+}(p)\neq0$.
\end{lemma}

\begin{proof}
Since $p\equiv1\pmod{12}$, therefore we can write
\begin{align*}
    p=x^2+y^2,\quad x,y\in\mathbb Z,\quad x\text{ is odd,}\\
    p=z^2+3w^2,\quad z,w\in\mathbb Z,\quad z\text{ is odd.}
\end{align*}
Thus,
\begin{align*}
    t_{1+}(p)=(+/-)\left((x+iy)^{12}+(x-iy)^{12}\right),\quad t_{2+}(p)=(+/-)\left((z+w\sqrt{-3})^{12}+(z-w\sqrt{-3})^{12}\right).
\end{align*}
We have already observed that the sign ($+$ or $-$) can be safely ignored as the same sign appears on both $t_{1+}(p)$ and $t_{2+}(p)$.

Now,
\begin{align}
    t_{1+}(p)&=2 p^6-144 p^5 x^2+1680 p^4 x^4-7168 p^3 x^6+13824 p^2 x^8-12288 p x^{10}+4096 x^{12},\label{eq:13}\\
    t_{2+}(p)&=2 p^6-144 p^5 z^2+1680 p^4 z^4-7168 p^3 z^6+13824 p^2 z^8-12288 p z^{10}+4096 z^{12}.\label{eq:14}
\end{align}
So,
\begin{align*}
    t_{1+}(p)-t_{2+}(p)=&-144 p^5 (x^2-z^2)+1680 p^4 (x^4-z^4)-7168 p^3 (x^6-z^6)+13824 p^2 (x^8-z^8)\\
    &-12288 p (x^{10}-z^{10})+4096 (x^{12}-z^{12})
\end{align*}
Clearly, on the RHS, the highest power of $2$ dividing the first term is strictly less than the highest power of $2$ dividing each of the other terms. Hence, $ord_2\left(t_{1+}(p)-t_{2+}(p)\right)<\infty$ and so $t_{1+}(p)-t_{2+}(p)\neq0$.

\end{proof}

\begin{lemma}\label{lemma2}
For a prime $p\equiv1\pmod{12}$ and $\alpha>1$, we have $ord_2\left(t_{1+}(p^\alpha)-t_{2+}(p^\alpha)\right)<\infty$ and so $t_{1+}(p^\alpha)-t_{2+}(p^\alpha)\neq0$.
\end{lemma}

\begin{proof}
We write $t_{1+}(p)=2a, t_{2+}(p)=2b$,  where $a, b$ are odd integers and $a\neq b$. First we assume, $\alpha=2\beta+1$ for some $\beta\geq1$. By (\ref{eq:2}), we have
\begin{align}\label{eq:11}
    t_{1+}(p^{2\beta+1})-t_{2+}(p^{2\beta+1})&=\left(t_{1+}(p)^{2\beta+1}-t_{2+}(p)^{2\beta+1}\right)-\binom{2\beta}{1}p^{12}\left(t_{1+}(p)^{2\beta-1}-t_{2+}(p)^{2\beta-1}\right)\\\nonumber
    &\quad+\dots+(-1)^\beta\binom{\beta+1}{1}p^{12\beta}\left(t_{1+}(p)-t_{2+}(p)\right)\\\nonumber
    &=\binom{\beta+(\beta+1)}{2(\beta+1)-1}2^{2(\beta+1)-1}(a^{2\beta+1}-b^{2\beta+1})-\binom{\beta+\beta}{2\beta-1}2^{2\beta-1}p^{12}(a^{2\beta-1}-b^{2\beta-1})\\\nonumber
    &\quad+\dots+(-1)^\beta(\beta+1)2p^{12\beta}(a-b).
\end{align}
Let $ord_2(2(\beta+1))=\gamma$. For any $\kappa=2,3,\dots,\beta+1$, we have
\begin{align*}
    \binom{\beta+\kappa}{2\kappa-1}2^{2\kappa-1}&=\frac{(\beta+\kappa)\cdots(\beta+1)\beta(\beta-1)\cdots(\beta-(\kappa-2))}{(2\kappa-1)!}\cdot2^{2\kappa-1}\\
    &=\frac{(\beta+\kappa)\cdots(\beta+1)\beta(\beta-1)\cdots(\beta-\kappa+2)}{u}\cdot2^{2\kappa-1-s},
\end{align*}
where $u$ is odd and $s=(\kappa-1)+\floor*{\frac{\kappa-1}{2}}+\dots+\floor*{\frac{\kappa-1}{2^n}}$, $n$ being the largest non-negative integer such that $2^n\leq \kappa-1$. Now,
\begin{align*}
    2\kappa-1-s&=2\kappa-1-\left((\kappa-1)+\floor*{\frac{\kappa-1}{2}}+\dots+\floor*{\frac{\kappa-1}{2^n}}\right)\\
    &\geq2\kappa-1-\left((\kappa-1)+{\frac{\kappa-1}{2}}+\dots+{\frac{\kappa-1}{2^n}}\right)\\
    &=1+\frac{\kappa-1}{2^n}\\
    &\geq2.
\end{align*}
So, we have
\begin{align*}
    \binom{\beta+\kappa}{2\kappa-1}2^{2\kappa-1}&=\frac{(\beta+\kappa)\cdots(2(\beta+1))\beta(\beta-1)\cdots(\beta-\kappa+2)}{u}\cdot2^{2\kappa-2-s}\\
    &=2^{\gamma+1}M.
\end{align*}
Now from (\ref{eq:11}), $ord_2\left(t_{1+}(p^\alpha)-t_{2+}(p^\alpha)\right)<\infty$ and so $t_{1+}(p^\alpha)-t_{2+}(p^\alpha)\neq0$.

Next, we assume, $\alpha=2\beta$ for some $\beta\geq1$. By (\ref{eq:1}), we have
\begin{align}\label{eq:12}
    t_{1+}(p^{2\beta})-t_{2+}(p^{2\beta})&=\left(t_{1+}(p)^{2\beta}-t_{2+}(p)^{2\beta}\right)-\binom{2\beta-1}{1}p^{12}\left(t_{1+}(p)^{2\beta-2}-t_{2+}(p)^{2\beta-2}\right)\\\nonumber
    &\quad+\dots+(-1)^{\beta-1}\binom{\beta+1}{2}p^{12\beta-12}\left(t_{1+}(p)^2-t_{2+}(p)^2\right)\\\nonumber
    &=\binom{\beta+\beta}{2\beta}2^{2\beta}(a^{2\beta}-b^{2\beta})-\binom{\beta+(\beta-1)}{2(\beta-1)}2^{2\beta-2}(a^{2\beta-2}-b^{2\beta-2})\\\nonumber
    &\quad+\dots+(-1)^{\beta-1}\binom{\beta+1}{2}2^2p^{12\beta-12}(a^2-b^2).
\end{align}
If $\beta=1$, then
\begin{align*}
    t_{1+}(p^2)-t_{2+}(p^2)=4(a^2-b^2)\neq0.
\end{align*}
Moreover, $ord_2\left(t_{1+}(p^2)-t_{2+}(p^2)\right)<\infty$. We assume, $\beta\geq2$. Note that $\binom{\beta+1}{2}2^2=2\beta(\beta+1)$. Let $ord_2(2\beta(\beta+1))=\delta$. For any $\kappa=2,3,\dots,\beta$, we have
\begin{align*}
    \binom{\beta+\kappa}{2\kappa}2^{2\kappa}&=\frac{(\beta+\kappa)\cdots(\beta+1)\beta(\beta-1)\cdots(\beta-\kappa+1)}{(2\kappa)!}\cdot2^{2\kappa}\\
    &=\frac{(\beta+\kappa)\cdots(\beta+1)\beta(\beta-1)\cdots(\beta-\kappa+1)}{v}\cdot2^{2\kappa-s'},
\end{align*}
where $v$ is odd and $s'=\kappa+\floor*{\frac{\kappa}{2}}+\dots+\floor*{\frac{\kappa}{2^n}}$, $n$ being the largest non-negative integer such that $2^n\leq \kappa$. Now,
\begin{align*}
    2\kappa-s'&=2\kappa-\left(\kappa+\floor*{\frac{\kappa}{2}}+\dots+\floor*{\frac{\kappa}{2^n}}\right)\\
    &\geq2\kappa-\left(\kappa+\frac{\kappa}{2}+\dots+\frac{\kappa}{2^n}\right)\\
    &=\frac{\kappa}{2^n}\\
    &\geq1.
\end{align*}
So, we have
\begin{align*}
    \binom{\beta+\kappa}{2\kappa}2^{2\kappa}&=\frac{(\beta+\kappa)\cdots(2(\beta+1)\beta)(\beta-1)\cdots(\beta-\kappa+1)}{v}\cdot2^{2\kappa-s'-1}\\
    &=2^{\delta+(\kappa-1)}N.
\end{align*}
But $\kappa-1\geq1$. Now from (\ref{eq:12}), $ord_2\left(t_{1+}(p^\alpha)-t_{2+}(p^\alpha)\right)<\infty$ and so $t_{1+}(p^\alpha)-t_{2+}(p^\alpha)\neq0$.
\end{proof}

\begin{lemma}\label{lemma3}
For a prime $p\equiv5\pmod{12}$ and $\alpha>1$ even, we have $ord_2\left(t_{1+}(p^\alpha)-t_{2+}(p^\alpha)\right)<\infty$ and so $t_{1+}(p^\alpha)-t_{2+}(p^\alpha)\neq0$.
\end{lemma}

\begin{proof}
Let $\alpha=2\beta$. In view of (\ref{eq:16}), we can write $t_{2+}(p)^2=2c$ for some integer $c$. By (\ref{eq:1}) and (\ref{eq:9}),  we have
\begin{align*}
    t_{1+}(p^{2\beta})-&t_{2+}(p^{2\beta})\\
    &=-\left(t_{2+}(p)^{2\beta}-\binom{2\beta-1}{1}p^{12}t_{2+}(p)^{2\beta-2}+\dots+(-1)^{\beta-1}\binom{\beta+1}{2}p^{12\beta-12}t_{2+}(p)^2\right)\\
    &=-\left(\binom{\beta+\beta}{2\beta}2^{2\beta}c^{2\beta}-\binom{\beta+(\beta-1)}{2(\beta-1)}2^{2\beta-2}c^{2\beta-2}+\dots+(-1)^{\beta-1}\binom{\beta+1}{2}2^2p^{12\beta-12}c^2\right).
\end{align*}
Now we can finish the proof by proceeding with the same argument as in the case of $\alpha$ being even in the proof of Lemma \ref{lemma2}.
\end{proof}

\begin{lemma}\label{lemma4}
For a prime $p\equiv7\pmod{12}$ and $\alpha>1$ even, we have $ord_2\left(t_{1+}(p^\alpha)-t_{2+}(p^\alpha)\right)<\infty$ and so $t_{1+}(p^\alpha)-t_{2+}(p^\alpha)\neq0$.
\end{lemma}

\begin{proof}
Let $\alpha=2\beta$. In view of (\ref{eq:15}), we can write $t_{1+}(p)^2=2d$ for some integer $d$. By (\ref{eq:3}) and (\ref{eq:7}),  we have
\begin{align*}
    t_{1+}(p^{2\beta})-t_{2+}(p^{2\beta})&=t_{1+}(p)^{2\beta}+\binom{2\beta-1}{1}p^{12}t_{1+}(p)^{2\beta-2}+\dots+\binom{\beta+1}{2}p^{12\beta-12}t_{1+}(p)^2\\
    &=\binom{\beta+\beta}{2\beta}2^{2\beta}d^{2\beta}+\binom{\beta+(\beta-1)}{2(\beta-1)}2^{2\beta-2}d^{2\beta-2}+\dots+\binom{\beta+1}{2}2^2p^{12\beta-12}d^2.
\end{align*}
Now we can finish the proof by proceeding with the same argument as in the case of $\alpha$ being even in the proof of Lemma \ref{lemma2}.
\end{proof}

\section{\textbf{Proof of the Main Results}}
In this section we first investigate some partial cases of Statement\ref{negation statement}. Then as applications of those results we prove Theorem \ref{MT 1} and Theorem \ref{MT 2}.
\begin{theorem}
Let $12n+13=p^{\alpha}$, where $p$ is a prime such that $p\not \equiv 11 \pmod {12}$. Then $p_{26}(n)\ne 0$.
\end{theorem}
\begin{proof}
If $p\equiv 1 \pmod {12}$, then the result follows from Lemma~\ref{lemma1} and Lemma \ref{lemma2}. On the other hand when $p \equiv 5 \pmod {12}$ (resp., $p\equiv 7 \pmod {12}$), then the result follows from Lemma \ref{lemma3} (resp., Lemma \ref{lemma4}). Note that $\alpha$ is odd when $p\equiv 5$ or $7 \pmod {12}$.
\end{proof}

\begin{theorem}
\label{power of primes 5 mod 12 is odd}
Let for every prime number $p\equiv3\pmod4$, $\ord_p(12n+13)$ is even; and there exists a prime number $p'\equiv5\pmod{12}$ such that $\ord_{p'}(12n+13)$ is odd. Then $p_{26}(n)\neq 0$.
\end{theorem}
\begin{proof}
By (\ref{eq:8}) and (\ref{eq:9}), it is clear that $t_{1\pm}(12n+13)=0$. Now, by (\ref{eq:8}), (\ref{eq:6}), (\ref{eq:7}) and (\ref{eq:10}), we have
\begin{align*}
    t_{2+}(12n+13)+t_{2-}(12n+13)=\prod_{p\not\equiv5(12)}t_{2+}(p^{\alpha_p})\left(\prod_{p\equiv5(12)}t_{2+}(p^{\alpha_p})+\prod_{p\equiv5(12)}t_{2-}(p^{\alpha_p})\right).
\end{align*}
Also note that if $p\equiv5\pmod{12}$ then $t_{2+}(p^\alpha)=t_{2-}(p^\alpha)$ for any even $\alpha$ and $t_{2+}(p^\alpha)=-t_{2-}(p^\alpha)$ for any odd $\alpha$. Thus, we have
\begin{align*}
    \prod_{p\equiv5(12)}t_{2+}(p^{\alpha_p})+\prod_{p\equiv5(12)}t_{2-}(p^{\alpha_p})&=\prod_{\substack{p\equiv5(12)\\\alpha_p \text{ even}}}t_{2+}(p^{\alpha_p})\left(\prod_{\substack{p\equiv5(12)\\\alpha_p \text{ odd}}}t_{2+}(p^{\alpha_p})+\prod_{\substack{p\equiv5(12)\\\alpha_p \text{ odd}}}t_{2-}(p^{\alpha_p})\right)\\
    &=(1+(-1)^\mu)\prod_{\substack{p\equiv5(12)\\\alpha_p \text{ even}}}t_{2+}(p^{\alpha_p})\prod_{\substack{p\equiv5(12)\\\alpha_p \text{ odd}}}t_{2+}(p^{\alpha_p}),
\end{align*}
where $\mu$ is the number of odd $\alpha_p$ for $p\equiv5\pmod{12}$. Since $\alpha_p$ is even for all $p\equiv3\pmod4$, therefore we must have $\mu$ even because otherwise we will have $12n+13\equiv5\pmod{12}$, which is impossible. In this case, clearly $p_{26}(n)\neq0$.
\end{proof}

\begin{theorem}
\label{power of prime 5 mod 12 is even}
Let $\dis 12n+13=\prod_{p \equiv 1 \pmod {12}}p^{\alpha_p}\prod_{p \equiv 5 \pmod {12}}p^{2\beta_p}\prod_{p \equiv 7 \pmod {12}}p^{2\gamma_p}\prod_{p \equiv 11 \pmod {12}}p^{2\delta_p}$ such that $25|12n+13$ and $\alpha_p\not\equiv 4 \pmod 5$. Then $p_{26}(n)\ne 0$.
\end{theorem}
\begin{proof}

By the assumption on $12n+13$ we have $t_{1+}(12n+13)=t_{1-}(12n+13)$ and $t_{2+}(12n+13)=t_{2-}(12n+13)$.
From Proposition \ref{t2+, t1+ prime = 1 mod 12}, Proposition \ref{t2+ prime = 5 mod 12}, \eqref{eq:7} and \eqref{eq:6} we get 
$$5\nmid \prod_{p \equiv 1 \pmod {12}}t_{2+}(p^{\alpha_p})\prod_{p \equiv 5 \pmod {12}}t_{2+}(p^{2\beta_p})\prod_{p \equiv 7 \pmod {12}}t_{2+}(p^{2\gamma_p})\prod_{p \equiv 11 \pmod {12}}t_{2+}(p^{2\delta_p}),$$
i.e $5\nmid t_{2+}(12n+13)$.
 On the other hand, from \eqref{eq:9} we have $5|t_{1+}(5^{2\beta}) $ for $\beta\in \mathbb{N}$. Since $25|12n+13$, we obtain $5|t_{1+}(12n+13)$. Hence we conclude that $5\nmid (t_{1+}(12n+13)-t_{2+}(12n+13))$. Equivalently, $5\nmid p_{26}(n)$. Thus $p_{26}(n)\ne 0$.
\end{proof}

\begin{theorem}
\label{power of prime 7 mod 12 is even}
Let $\dis 12n+13=\prod_{p \equiv 1 \pmod {12}}p^{\alpha_p}\prod_{p \equiv 5 \pmod {12}}p^{2\beta_p}\prod_{p \equiv 7 \pmod {12}}p^{2\gamma_p}\prod_{p \equiv 11 \pmod {12}}p^{2\delta_p}$ such that $49|12n+13$ and $\alpha_p\not\equiv 6 \pmod 7$. Then $p_{26}(n)\ne 0$.
\end{theorem}
\begin{proof}
The proof is similar to the proof of Theorem \ref{power of prime 5 mod 12 is even}.
By the assumption on $12n+13$ we have $t_{1+}(12n+13)=t_{1-}(12n+13)$ and $t_{2+}(12n+13)=t_{2-}(12n+13)$.
 From Proposition \ref{t2+, t1+ prime = 1 mod 12}, \eqref{eq:9}, Proposition \ref{t1+  prime = 7 mod 12}, \eqref{eq:6} we get 
$$7\nmid \prod_{p \equiv 1 \pmod {12}}t_{1+}(p^{\alpha_p})\prod_{p \equiv 5 \pmod {12}}t_{1+}(p^{2\beta_p})\prod_{p \equiv 7 \pmod {12}}t_{1+}(p^{2\gamma_p})\prod_{p \equiv 11 \pmod {12}}t_{1+}(p^{2\delta_p}),$$
i.e $7\nmid t_{1+}(12n+13)$. On the other hand, since $7|12n+13$, from \eqref{eq:7} we have $7|t_{2+}(12n+13)$. Hence we conclude that $7\nmid (t_{1+}(12n+13)-t_{2+}(12n+13))$. Consequently, $p_{26}(n)\ne 0$.
\end{proof}

\begin{theorem}
\label{power of primes 7 mod 12 is odd}
Let for every prime number $p\equiv2\pmod3$, $\ord_p(12n+13)$ is even; and there exists a prime number $p'\equiv7\pmod{12}$ such that $\ord_{p'}(12n+13)$ is odd. Then $p_{26}(n)\ne 0$.
\end{theorem}
\begin{proof}
By (\ref{eq:8}) and (\ref{eq:7}), it is clear that $t_{2\pm}(12n+13)=0$. Now, by (\ref{eq:8}), (\ref{eq:6}), (\ref{eq:9}) and (\ref{eq:17}), we have
\begin{align*}
    t_{1+}(12n+13)+t_{1-}(12n+13)=\prod_{p\not\equiv7(12)}t_{1+}(p^{\alpha_p})\left(\prod_{p\equiv7(12)}t_{1+}(p^{\alpha_p})+\prod_{p\equiv7(12)}t_{1-}(p^{\alpha_p})\right).
\end{align*}
Also note that if $p\equiv7\pmod{12}$ then $t_{1+}(p^\alpha)=t_{1-}(p^\alpha)$ for any even $\alpha$ and $t_{1+}(p^\alpha)=-t_{1-}(p^\alpha)$ for any odd $\alpha$. Thus, we have
\begin{align*}
    \prod_{p\equiv7(12)}t_{1+}(p^{\alpha_p})+\prod_{p\equiv7(12)}t_{1-}(p^{\alpha_p})&=\prod_{\substack{p\equiv7(12)\\\alpha_p \text{ even}}}t_{1+}(p^{\alpha_p})\left(\prod_{\substack{p\equiv7(12)\\\alpha_p \text{ odd}}}t_{1+}(p^{\alpha_p})+\prod_{\substack{p\equiv7(12)\\\alpha_p \text{ odd}}}t_{1-}(p^{\alpha_p})\right)\\
    &=(1+(-1)^\nu)\prod_{\substack{p\equiv7(12)\\\alpha_p \text{ even}}}t_{1+}(p^{\alpha_p})\prod_{\substack{p\equiv7(12)\\\alpha_p \text{ odd}}}t_{1+}(p^{\alpha_p}),
\end{align*}
where $\nu$ is the number of odd $\alpha_p$ for $p\equiv7\pmod{12}$. Since $\alpha_p$ is even for all $p\equiv2\pmod3$, therefore we must have $\nu$ even because otherwise we will have $12n+13\equiv7\pmod{12}$, which is impossible. In this case, clearly $p_{26}(n)\neq0$.
\end{proof}

Finally, as applications of the previous results we now prove Theorem \ref{MT 1} and Theorem \ref{MT 2}. 

\subsection{Proof of Theorem \ref{MT 1}}
In this section we always assume that $n\in \mathbb{N}$ such that $\ord_p(12n+1)\not\equiv 4 \pmod 5$ for every prime $p\equiv 1 \pmod {12}$. 
Furthermore, if n is such that there exists a prime number $p\equiv 3 \pmod 4$ such that $\ord_p(12n+1)$ is odd (equivalently $\ord_p(12(25n+1)+13)$ is odd) and there exists a prime number $p^\prime \equiv 2 \pmod 3$ such that $\ord_{p^\prime}(12n+1)$ is odd (equivalently $\ord_{p^\prime}(12(25n+1)+13)$ is odd), then from Theorem \ref{Serre theorem} it follows that $p_{26}(25n+1)=0$. Hence it suffices to prove the converse part.
Note that the converse part of Theorem \ref{MT 1} is equivalent to prove the following theorem:

\begin{theorem}
If $n$ satisfies one of the following conditions, then $p_{26}(25n+1)\ne 0$.
\begin{enumerate}
\item For every prime $p\equiv 3 \pmod 4$, $\ord_p(12n+1)$ is even.
\item For every prime $p\equiv 2 \pmod 3$, $\ord_p(12n+1)$ is even.
\end{enumerate}
\end{theorem}

\begin{proof}
\begin{enumerate}
\item Let $\ord_p(12n+1)$ be even for every prime $p\equiv 3\pmod 4$. Hence $\ord_p(12(25n+1)+13)$ is even for every prime $p\equiv 3 \pmod 4$. If there exists a prime number $p'\equiv 5\pmod{12}$ such that $\ord_{p'}(12(25n+1)+13)$ is odd, then the result follows from Theorem \ref{power of primes 5 mod 12 is odd}. On the other hand, if $\ord_{p'}(12(25n+1)+13)$ is even for all prime $p^\prime \equiv 5 \pmod {12}$, then the result follows from Theorem \ref{power of prime 5 mod 12 is even}.
\item Let $\ord_p(12n+1)$ be even for every prime $p\equiv 2 \pmod 3$. Hence $\ord_p(12(25n+1)+13)$ is even for every prime $p\equiv 2 \pmod 3$. If there exists a prime $p^\prime \equiv 7 \pmod {12}$ such that $\ord(12(25n+1)+13)$ is odd, then the result follows from Theorem \ref{power of primes 7 mod 12 is odd}. On the other hand, if $\ord_{p'}(12(25n+1)+13)$ is even for all prime $p^\prime \equiv 7 \pmod {12}$, then the result follows from Theorem \ref{power of prime 5 mod 12 is even}.
\end{enumerate}
\end{proof}

\subsection{Proof of Theorem \ref{MT 2}}
Now assume that $n\in \mathbb{N}$ such that $\ord_p(12n+1)\not\equiv 6 \pmod 7$ for every prime $p\equiv 1 \pmod {12}$. The proof of Theorem \ref{MT 2} is analogous to the proof of Theorem \ref{MT 1}. For the sake of completeness we are giving the complete proof. 

If $n$ satisfies the assumptions of Theorem \ref{MT 2}, then from Theorem \ref{Serre theorem} we conclude that $p_{26}(49n+3)=0$.
Note that the converse part of Theorem \ref{MT 2} is equivalent to prove the following theorem:

\begin{theorem}
If $n$ satisfies one of the following conditions, then $p_{26}(49n+3)\ne 0$.
\begin{enumerate}
\item For every prime $p\equiv 3 \pmod 4$, $\ord_p(12n+1)$ is even.
\item For every prime $p\equiv 2 \pmod 3$, $\ord_p(12n+1)$ is even.
\end{enumerate}
\end{theorem}

\begin{proof}
\begin{enumerate}
\item Let $\ord_p(12n+1)$ be even for every prime $p\equiv 3\pmod 4$. Hence $\ord_p(12(49n+3)+13)$ is even for every prime $p\equiv 3 \pmod 4$. If there exists a prime number $p'\equiv 5\pmod{12}$ such that $\ord_{p'}(12(49n+3)+13)$ is odd, then the result follows from Theorem \ref{power of primes 5 mod 12 is odd}. On the other hand, if $\ord_{p'}(12(49n+3)+13)$ is even for all prime $p^\prime \equiv 5 \pmod {12}$, then the result follows from Theorem \ref{power of prime 7 mod 12 is even}.
\item Let $\ord_p(12n+1)$ be even for every prime $p\equiv 2 \pmod 3$. Hence $\ord_p(12(49n+3)+13)$ is even for every prime $p\equiv 2 \pmod 3$. If there exists a prime $p^\prime \equiv 7 \pmod {12}$ such that $\ord(12(49n+3)+13)$ is odd, then the result follows from Theorem \ref{power of primes 7 mod 12 is odd}. On the other hand, if $\ord_{p'}(12(49n+3)+13)$ is even for all prime $p^\prime \equiv 7 \pmod {12}$, then the result follows from Theorem \ref{power of prime 7 mod 12 is even}.
\end{enumerate}
\end{proof}

\section*{Acknowledgments}
The authors are grateful to Prof. J.P. Serre for informing about this open problem. The authors would
like thank Abinash Sarma for helpful comments and suggestions. This paper was written when the Tarun Dalal was a
visiting at IISER Thiruvananthapuram and he is grateful to the School of Mathematics
for the excellent support and hospitality. Srilakshmi Krishnamoorthy's research was partially supported by the SERB
grant CRG/2023/009035

\end{document}